\documentclass[11pt,a4paper,reqno]{amsart}
\usepackage{amsfonts}
\usepackage{amssymb}
\usepackage{hyperref}
\usepackage{color}

\numberwithin{equation}{section}
\newtheorem{theorem}{Theorem}[section]
\newtheorem{lemma}[theorem]{Lemma}
\newtheorem{remark}[theorem]{Remark}

\newtheorem{conjecture}[theorem]{Conjecture}
\allowdisplaybreaks

\def\enddoc{\end{document}}

\def\FRAME#1#2#3#4#5#6#7#8
{
 \begin{figure}[ptbh]
 \begin{center}
 \includegraphics[height=#3]{#7}
 \caption{#5}
 \label{#6}
 \end{center}
 \end{figure}
}

\begin{document}
\author{Qingsong Gu}
\address{Department of mathematics, Nanjing University, Nanjing 210093, P. R. China} \email{qingsonggu@nju.edu.cn}
\author{Xueping Huang}
\address{Department of Mathematics, Nanjing University of Information Science and Technology,
Nanjing 210044, P. R. China}
\email{hxp@nuist.edu.cn}
\author{Yuhua Sun}
\address{School of Mathematical Sciences and LPMC, Nankai University, 300071
Tianjin, P. R. China}
\email{sunyuhua@nankai.edu.cn}

\title{Superlinear elliptic inequalities on weighted graphs}
\thanks{\noindent Qingsong Gu was supported by the National Natural Science Foundation of
China (Grant No.12101303). Xueping Huang was supported by the National Natural Science Foundation of China (Grant
No. 11601238) and The Startup Foundation for Introducing Talent of Nanjing University of
Information Science and Technology (Grant No. 2015r053). Yuhua Sun was supported by the National Natural Science Foundation of
China (Grant No.11501303, No.11871296), and Tianjin Natural Science Foundation (Grant No.19JCQNJC14600). 
}

\subjclass[2010]{Primary 35J61; Secondary 58J05, 31B10, 42B37}
\keywords{Semi-linear elliptic inequality, weighted graph, critical exponent, volume growth}

\begin{abstract}
Let $(V,\mu)$ be an infinite, connected, locally finite weighted graph. We study the problem of existence or non-existence of
positive solutions to a semi-linear elliptic inequality
\begin{equation*}
\Delta u+u^{\sigma}\leq0\quad \text{in}\,\,V,
\end{equation*}
where $\Delta$ is the standard graph Laplacian on $V$ and $\sigma>0$. For $\sigma\in(0,1]$, the inequality admits no nontrivial positive solution. For $\sigma>1$, assuming condition \textbf{($p_0$)} on $(V,\mu)$, we obtain a sharp condition for nonexistence of
positive solutions in terms of the volume growth of the graph, that is
\begin{equation*}
\mu(o,n)\lesssim n^{\frac{2\sigma}{\sigma-1}}(\ln n)^{\frac{1}{\sigma-1}}
 \end{equation*}
 for some $o\in V$ and all large enough $n$. For any $\varepsilon>0$, we can construct an example on a homogeneous tree $\mathbb T_N$ with $\mu(o,n)\approx n^{\frac{2\sigma}{\sigma-1}}(\ln n)^{\frac{1}{\sigma-1}+\varepsilon}$, and a solution to the inequality on $(\mathbb T_N,\mu)$ to illustrate the sharpness
 of $\frac{2\sigma}{\sigma-1}$ and $\frac{1}{\sigma-1}$.
\end{abstract}

\maketitle
\tableofcontents

\section{Introduction}

Throughout the paper, let $G=(V,E)$ be an infinite, connected, locally finite graph, where $V$ denotes the vertex set, and $E$ denotes the edge set. We allow only at most one edge for any two distinct vertices, and do not allow any edges from a vertex to itself.

Let $\mu:\ E\rightarrow[0,\infty)$ be a weight defined on its edges, where $\mu_{xy}>0$ if and only if $(x,y)\in E$ and $\mu_{xy}=\mu_{yx}$. Such graph $(V,E, \mu)$ is called a weighted graph. Since $\mu$ contains all the information of the edge set $E$, we just denote the weighted graph as $(V,\mu)$.

For a vertex $x\in V$, denote $y\sim x$ if there is an edge between $x$ and $y$, and we define $\text{deg}(x):\ =\#\{y\in V:\ y\sim x\}$ to be the degree of $x$. Let
$$\mu(x)=\sum_{y\sim x}\mu_{xy}$$
be the measure of $x$. Let $\ell(V)$ be the collection of all functions on $V$.
On the graph $(V,E)$, for any two vertices $x$ and $y$, we define $d(x,y)$ to be the minimal number of edges among all possible paths connecting $x$ and $y$ in $V$. Then $d(\cdot,\cdot)$ is a distance on $V$, we call it the graph distance.

Fix an arbitrary vertex $o\in V$. For an integer $n\geq1$, denote by $B(o,n):=\{x\in V:\ d(o,x)\leq n\}$ the closed ball in $V$ with center $o$ and radius $n$. Let us denote
\begin{equation}\label{Vol}
\mu(B(o,n))=\sum_{x\in B(o,n)}\mu(x),
\end{equation}
and the Laplace operator $\Delta$ by
\begin{equation*}
\Delta u(x)=\frac{1}{\mu(x)}\sum_{y\sim x}\mu_{xy}(u(y)-u(x)),\qquad\qquad\mbox{for $u\in\ell(V)$}.
 \end{equation*}

In this paper we are concerned with the following problem: characterize $
\sigma>0$ and $\mu(o,n)$ for which there exists a positive
solution $u$ to the following superlinear elliptic
inequality:
\begin{equation}
\Delta u+u^\sigma\leq0\quad \text{in}\,\,V.  \label{Elliptic}
\end{equation}

Recall that a weighted graph $(V,\mu)$ is called {\it recurrent} if the random walk on $(V,\mu)$ with one-step transition probability $P(x,y)=\frac{\mu_{xy}}{\mu(x)}$ is recurrent; otherwise, $(V,\mu)$ is called {\it transient}. It is known that $(V,\mu)$ is recurrent if and only if  any nonnegative superharmonic function (i.e. $\Delta u\leq0$) on $(V,\mu)$ is identically equal to a constant, which is termed as {\it parabolic}.

The notion of parabolicity of graph can be regarded as a generalization of the parabolicity of manifolds.
Let $(M, \mu)$ be a geodesically complete noncompact manifold with Riemannian measure $\mu$. Recall that in a celebrated paper of Cheng and Yau \cite{CY}, it is proved that if for some $o\in M$
\begin{eqnarray*}
\mu(B(o,r))\lesssim r^2\qquad\mbox{for all large enough $r$},
\end{eqnarray*}
then $M$ is parabolic.

Later, three authors Grigor'yan, Karp, Varopoulos \cite{G85,K,V} independently showed that if
\begin{eqnarray}\label{vol-int}
\int^{\infty}\frac{r}{\mu(B(o,r))}dr=\infty,
\end{eqnarray}
then $M$ is parabolic.

The following Nash-Williams' test for parabolicity (\ref{vol-int}) on weighted graphs is well known (e.g. \cite{W}): if
\begin{eqnarray}\label{votest}
\sum^{\infty}_{n=1}\frac{n}{\mu(B(o,n))}=\infty,
\end{eqnarray}
then the graph $(V,\mu)$ is recurrent.

Consequently, under the assumption of \eqref{votest}, then the only nonnegative solution to \eqref{Elliptic} is identically zero.
The above implies that \eqref{votest} is far from sharpness for the nonexistence result to \eqref{Elliptic}.

 In order to obtain the sharp condition,
we need some additional restriction for the graph $(V,\mu)$ for different $\sigma$. Here are our main results.

\begin{theorem}\label{result-1}
Assume $\sigma\in(0,1]$. If $u$ is a nonnegative solution to \eqref{Elliptic}, then $u\equiv0$.
\end{theorem}

 We say condition \textbf{($p_0$)} is satisfied on $(V, \mu)$:
if there exists $p_0>1$ such that for any $x\sim y$ in $V$,
\begin{equation*}
\frac{\mu_{xy}}{\mu(x)}\geq \frac{1}{p_0}.
\end{equation*}

Under condition \textbf{($p_0$)}, our nonexistence result to \eqref{Elliptic} is as follows.
\begin{theorem}\label{mainresult}
Let $\sigma>1$. Assume condition \textbf{$(p_0)$} is satisfied on $(V, \mu)$. If for some $o\in V$
\begin{equation}\label{volume}
\mu(B(o,n))\lesssim n^{\frac{2\sigma}{\sigma-1}}(\ln n)^{\frac{1}{\sigma-1}} \qquad \text{ for all large enough $n$}.
\end{equation}
Then the only nonnegative solution to \eqref{Elliptic} is identically zero.
\end{theorem}

Here the notation $``H\lesssim K"$ in the above and below means that $H\leq cK$ for some constant $c>0$, and $H\thickapprox K$ means that
both $H\lesssim K$ and $K\lesssim H$ hold.

\begin{remark}
The volume growth condition \eqref{volume} is similar to that on Riemannian manifolds given in \cite{GS1}: On a geodesically complete
noncompact manifold $M$, if for some $o\in M$
\begin{eqnarray}\label{vol-M}
\mu(B(o,r))\lesssim r^{\frac{2\sigma}{\sigma-1}}(\ln r)^{\frac{1}{\sigma-1}},\quad\mbox{for all large enough $r$},
\end{eqnarray}
then the only nonnegative solution to $\Delta u+u^{\sigma}\leq0$ with $\sigma>1$ on $M$ is identically zero.

Moreover, if the volume doubling condition and Poincar\'{e} inequality are satisfied on $M$, the above nonexistence results are still valid provided that (\ref{vol-M}) is weakened by
\begin{eqnarray*}
\int^{\infty}\frac{r^{2q-1}}{\mu(B(o, r))^{q-1}}dr=\infty,
\end{eqnarray*}
see \cite[Corollary 1.2]{GSV}.
\end{remark}
We do not know whether condition {\bf($p_0$)} can be removed.
The following conjecture is motivated by \cite[Conjecture 1]{GSV}.
\begin{conjecture}
On a weighted graph $(V,\mu)$, if
\begin{equation}\label{volumeconj}
\sum_{n=1}^{\infty}\frac{n^{2\sigma-1}}{\mu(B(o,n))^{\sigma-1}}=\infty,
\end{equation}
then the only nonnegative solution to \eqref{Elliptic} is identically zero.
\end{conjecture}

Besides the above nonexistence result (or called Liouville theorems) for elliptic equation on graphs, a lot of attention has been paid to different types of elliptic equations on graphs, see \cite{CM, Ge, GHJ, GLY1,GLY2, GLY3, HSZ, HS, HWY, LiuY}.
there are also many literature devoted to the parabolic equations on graphs, see \cite{LW1, LW2, LY}.

The methods of the proofs of Theorems \ref{result-1} and \ref{mainresult} are motivated by the one used in \cite{GS1}.
For the proof of Theorem \ref{result-1}, the case of $\sigma=1$ is trivial, while for the case of $0<\sigma<1$, we use the iteration method to derive the nonexistence result.
For the proof of Theorem \ref{mainresult}, we first discretize the test functions used in \cite{GS1}, and then apply the estimates in terms of test functions to derive that $u\in L^{\sigma}(d\mu)$ and finally obtain $u\equiv0$ by choosing a sequence of suitable test functions. However, different from the manifold, in the graph case, in order to obtain the necessary inequalities, we need a simple a priori property of the solution $u$: the ratio of the values of $u$ on any two neighboring vertices is uniformly bounded by $p_0$, where $p_0$ is the same as that in condition {\bf($p_0$)}. We state this property as Lemma \ref{th2.1}. We remark that it is an interesting question that whether the conclusion in Theorem \ref{mainresult} still hold if we remove the condition {\bf($p_0$)}.

The structure of this paper is as follows. In Section \ref{sec2} we give the statement of Theorem \ref{mainresult}, the main result of this paper; we also give Lemma \ref{th2.1} which will be used frequently in the proof of Theorem \ref{mainresult}. In Section \ref{sec3}, we provide a proof of Theorem \ref{mainresult}; we also construct an example on $\mathbb Z$ to illustrate the sharpness of the volume growth condition in Theorem \ref{mainresult}.

\section{Proof of Theorem \ref{result-1}}\label{sec2}

\begin{lemma}\label{lem-positive}
If $u$ is a nonnegative solution to (\ref{Elliptic}), then either $u>0$ or $u\equiv0$.
\end{lemma}
\begin{proof}
Assume that there exists some point $x_0$ such that $u(x_0)=0$, now we show that $u\equiv0$. Since from (\ref{Elliptic})
\begin{equation*}
0\geq\sum_{y\sim x_0}\frac{\mu_{x_0y}}{\mu(x_0)}u(y)-u(x_0)+u(x_0)^{\sigma}=\sum_{y\sim x_0}\frac{\mu_{x_0y}}{\mu(x_0)}u(y),
\end{equation*}
which implies that $u(y)=0$ for any $y\sim x$. Hence $u\equiv0$ by the connectedness of $V$.
\end{proof}

\begin{proof}[Proof of Theorem \ref{result-1}]
{\bf Case of $\sigma=1$.} In this case, inequality \eqref{Elliptic} is just
\begin{equation*}
\sum_{y\sim x}\frac{\mu_{xy}}{\mu(x)}u(y)-u(x)+u(x)=\sum_{y\sim x}\frac{\mu_{xy}}{\mu(x)}u(y)\leq0,
\end{equation*}
which implies $u(y)=0$ for any $y\sim x$. This implies $u\equiv0$ by the connectedness of $V$.

\bigskip

{\bf Case of $0<\sigma<1$.} We first show that if $u\not\equiv0$, then $u(x)\geq1$ for any $x\in V$.
By Lemma \ref{lem-positive}, we obtain that $u>0$ in $V$, by \eqref{Elliptic}, we have
\begin{equation*}
0\geq\sum_{y\sim x}\frac{\mu_{xy}}{\mu(x)}u(y)-u(x)+u(x)^{\sigma}\geq -u(x)+u(x)^{\sigma},
\end{equation*}
which implies $u(x)\geq1$ by using that $0<\sigma<1$.

Now fix any $x_0\in V$. Let $x_1$ be a neighboring vertex of $x_0$ such that $u(x_1)=\min_{y\sim x_0}u(y)$. Then by \eqref{Elliptic},
\begin{equation*}
u(x_0)-u(x_0)^{\sigma}\geq\sum_{y\sim x_0}\frac{\mu_{x_0y}}{\mu(x_0)}u(y)\geq\sum_{y\sim x_0}\frac{\mu_{x_0y}}{\mu(x_0)}u(x_1)=u(x_1),
\end{equation*}
that is
\begin{equation*}
u(x_1)\leq u(x_0)\left(1-u(x_0)^{\sigma-1}\right).
\end{equation*}
Inductively, for $x_i$, we can find $x_{i+1}\in V$ such that $u(x_{i+1})=\min_{y\sim x_i}u(y)$ and
\begin{equation}\label{eqinductive1}
u(x_{i+1})\leq u(x_i)\left(1-u(x_i)^{\sigma-1}\right)\qquad i\geq1.
\end{equation}
Since $u(x_{i})\leq u(x_{i-1})\leq\cdots\leq u(x_0)$, we have
\begin{equation}\label{eqinductive2}
1-u(x_{i})^{\sigma-1}\leq 1-u(x_0)^{\sigma-1}\qquad i\geq1.
\end{equation}
Multiplying both sides of \eqref{eqinductive1} over $i=0,1,\cdots,n-1$, then using \eqref{eqinductive2}, we obtain
\begin{equation}\label{eqinductive3}
u(x_n)\leq u(x_0)\prod_{i=0}^{n-1}\left(1-u(x_{i})^{\sigma-1}\right)\leq u(x_0)\left(1-u(x_0)^{\sigma-1}\right)^{n}.
\end{equation}
Since $1-u(x_0)^{\sigma-1}<1$, we see from \eqref{eqinductive3} that $u(x_n)\rightarrow0$ as $n\rightarrow0$, this contradicts the fact that
$u(x)\geq1$ for all $x\in V$. Therefore the only nonnegative solution to \eqref{Elliptic} is identically zero.
\end{proof}

\section{Proof of Theorem \ref{mainresult}}\label{sec3}


Before proceeding to the proof of Theorem \ref{mainresult}, we first give a simple property for nonnegative solutions to \eqref{Elliptic}
under condition \textbf{($p_0$)}.
\begin{lemma}\label{th2.1}
Let $\sigma>1$ and $(V,\mu)$ satisfy $(p_0)$. If $u\geq0$ is a solution to \eqref{Elliptic}, then either $u\equiv0$ or $0<u<1$ on $V$ and
\begin{equation}\label{compare}
\frac{1}{p_0}\leq\frac{u(x)}{u(y)}\leq p_0 \qquad \text{ for any $x\sim y$}.
\end{equation}
\end{lemma}
\begin{proof}
Assume $u\not\equiv0$, then by Lemma \ref{lem-positive}, we must have $u>0$ on $V$. Then by \eqref{Elliptic}, we have
\begin{equation*}
0\geq\frac{1}{\mu(x)}\sum_{y\sim x}\mu_{xy}(u(y)-u(x))+u(x)^\sigma>u(x)^\sigma-u(x),
\end{equation*}
which implies $u(x)<1$. Also we have
\begin{equation*}
0\geq\frac{1}{\mu(x)}\sum_{y\sim x}\mu_{xy}(u(y)-u(x))+u(x)^\sigma\geq\frac{\mu_{xy}}{\mu(x)}u(y)-u(x) \qquad \text{ for any $y\sim x$},
\end{equation*}
which implies
\begin{equation*}
\frac{u(x)}{u(y)}\geq\frac{\mu_{xy}}{\mu(x)}\geq \frac{1}{p_0},
\end{equation*}
by the property {\bf($p_0$)}. Exchanging $x$ and $y$, we get the other side inequality and hence \eqref{compare} holds.
\end{proof}

In the following, for brevity, we denote by $\nabla_{xy}f=f(y)-f(x)$ the discrete gradient of a function $f\in\ell (V)$ along a directed edge $x\rightarrow y$.

\begin{proof}[Proof of Theorem \ref{mainresult}] Assume $u>0$ is a solution to \eqref{Elliptic}. We will pick a sequence of test functions $\varphi_n\in \ell(V)$ which are compactly supported in $V$ and satisfy $0\leq\varphi_n\leq1$, then we obtain an estimate of the $L^\sigma$-norm of $u$ in terms of $\varphi_n$, and finally get $u\equiv0$, a contradiction. We separate the proof into the following two steps.

\bigskip

Step 1.
Let $s>2$ be a sufficiently large fixed number and $t\in(0,1)$ be sufficiently small. Let $\varphi\in\ell(V)$ be compactly supported in $V$, define $\psi=\varphi^su^{-t}$.
 From \eqref{Elliptic}, we have for any fixed $x\in V$,
 \begin{eqnarray*}
\sum_{y\in V}\mu_{xy}(\nabla_{xy}u)\psi(x)+\mu(x)u(x)^{\sigma}\psi(x)\leq0.
 \end{eqnarray*}
 Multiplying $\psi$ to the above and summing up over $x\in V$, we get
\begin{equation}\label{eq1}
\sum_{x,y\in V}\mu_{xy}(\nabla_{xy}u)\psi(x)+\sum_{x\in V}\mu(x)u(x)^{\sigma}\psi(x)\leq0.
\end{equation}
Then noticing
\begin{equation*}
\sum_{x,y\in V}\mu_{xy}(\nabla_{xy}u)\psi(x)=-\frac{1}{2}\sum_{x,y\in V}\mu_{xy}(\nabla_{xy}u)(\nabla_{xy}\psi),
\end{equation*}
and
\begin{equation*}
\nabla_{xy}\psi=\nabla_{xy}(\varphi^su^{-t})=u(y)^{-t}\nabla_{xy}(\varphi^s)+\varphi(x)^s\nabla_{xy}(u^{-t}),
\end{equation*}
we obtain
\begin{align*}
&\sum_{x,y\in V}\mu_{xy}(\nabla_{xy}u)\psi(x)\\
&=-\frac{1}{2}\sum_{x,y\in V}\mu_{xy}u(y)^{-t}(\nabla_{xy}u)\nabla_{xy}(\varphi^{s})-\frac{1}{2}\sum_{x,y\in V}\mu_{xy}\varphi(x)^{s}(\nabla_{xy}u)\nabla_{xy}(u^{-t}),
\end{align*}
Substituting this back into \eqref{eq1}, we have
\begin{align}\label{eq2}
-\frac{1}{2}\sum_{x,y\in V}\mu_{xy}\varphi(x)^{s}(\nabla_{xy}u)\nabla_{xy}(u^{-t})&+\sum_{x\in V}\mu(x)\varphi(x)^{s}u(x)^{\sigma-t}\notag\\
&\leq\frac{1}{2}\sum_{x,y\in V}\mu_{xy}u(y)^{-t}(\nabla_{xy}u)\nabla_{xy}(\varphi^{s}).
\end{align}

For the first term in the LHS of \eqref{eq2}, observe that by the mid-value theorem, there is $\xi$ between $u(y)$ and $u(x)$ such that
\begin{equation*}
\nabla_{xy}(u^{-t})=u(y)^{-t}-u(x)^{-t}=-t\xi^{-t-1}(u(y)-u(x))=-t\xi^{-t-1}\nabla_{xy}u.
\end{equation*}
Since $\xi$ is between two positive values $u(y)$ and $u(x)$, and also by \eqref{compare}, we must have $\frac{u(x)}{p_0}\leq \xi\leq u(x)p_0$.
Hence we have
\begin{align}
&-\frac{1}{2}\sum_{x,y\in V}\mu_{xy}\varphi(x)^{s}(\nabla_{xy}u)\nabla_{xy}(u^{-t})\notag\\
=&\frac{t}{2}\sum_{x,y\in V}\mu_{xy}\varphi(x)^{s}(\nabla_{xy}u)^2\xi^{-t-1}\notag\\
\geq&\frac{t}{2p_0^{t+1}}\sum_{x,y\in V}\mu_{xy}\varphi(x)^{s}u(x)^{-t-1}(\nabla_{xy}u)^2\notag\\
\geq&\frac{t}{2p_0^{2}}\sum_{x,y\in V}\mu_{xy}\varphi(x)^{s}u(x)^{-t-1}(\nabla_{xy}u)^2,\label{eq2'}
\end{align}
where we have used that $0<t<1$ and $p_0>1$.

For the RHS of \eqref{eq2}, similarly we have that there is some $\eta$ between the two nonnegative values $\varphi(x)$ and $\varphi(y)$ such that
\begin{equation}\label{eq3'}
\nabla_{xy}(\varphi^{s})=s\eta^{s-1}(\varphi(y)-\varphi(x))=s\eta^{s-1}\nabla_{xy}\varphi.
\end{equation}
By using Cauchy-Schwarz inequality, we obtain
\begin{align}
&\frac{1}{2}\sum_{x,y\in V}\mu_{xy}u(y)^{-t}(\nabla_{xy}u)\nabla_{xy}(\varphi^{s})\notag\\
=&\frac{s}{2}\sum_{x,y\in V}\mu_{xy}\left(u(y)^{-\frac{t+1}2}\eta^{\frac s2}\nabla_{xy}u\right)\left(u(y)^{-\frac{t-1}{2}}\eta^{\frac{s-2}{2}}\nabla_{xy}\varphi\right)\notag\\
\leq&\frac{\epsilon}{4}\sum_{x,y\in V}\mu_{xy}u(y)^{-t-1}\eta^{s}(\nabla_{xy}u)^2+\frac{s^2}{\epsilon}\sum_{x,y\in V}\mu_{xy}u(y)^{-t+1}\eta^{s-2}(\nabla_{xy}\varphi)^2,\label{eq3}
\end{align}
where $\epsilon>0$ will be determined later.

Combining \eqref{eq2} with \eqref{eq2'} and \eqref{eq3}, we arrive at
\begin{align}\label{eq3-0}
&\frac{t}{2p_0^{2}}\sum_{x,y\in V}\mu_{xy}\varphi(x)^{s}u(x)^{-t-1}(\nabla_{xy}u)^2+\sum_{x\in V}\mu(x)\varphi(x)^{s}u(x)^{\sigma-t}\notag\\
\leq&\frac{\epsilon}{4}\sum_{x,y\in V}\mu_{xy}u(y)^{-t-1}\eta^{s}(\nabla_{xy}u)^2+\frac{s^2}{\epsilon}\sum_{x,y\in V}\mu_{xy}u(y)^{-t+1}\eta^{s-2}(\nabla_{xy}\varphi)^2.
\end{align}

We first deal with the first term in the RHS of \eqref{eq3-0}. Observe that $\eta^s\leq (\max\{\varphi(x),\varphi(y)\})^s\leq\varphi(x)^s+\varphi(y)^s$ and by applying Lemma \ref{compare}, we have
\begin{align}
\sum_{x,y\in V}\mu_{xy}u(y)^{-t-1}\eta^{s}(\nabla_{xy}u)^2&\leq \sum_{x,y\in V}\mu_{xy}u(y)^{-t-1}(\varphi(x)^s+\varphi(y)^s)(\nabla_{xy}u)^2\notag\\
&\leq (1+p_0^{1+t})\sum_{x,y\in V}\mu_{xy}u(x)^{-t-1}\varphi(x)^s(\nabla_{xy}u)^2\notag\\
&\leq (1+p_0^{2})\sum_{x,y\in V}\mu_{xy}u(x)^{-t-1}\varphi(x)^s(\nabla_{xy}u)^2,\label{eq4}
\end{align}
where in the last step we used $0<t<1$ and $p_0>1$.

For the second term in the RHS of \eqref{eq3-0}, since $s>2$, similar to the above, we have
\begin{equation}
\sum_{x,y\in V}\mu_{xy}u(y)^{-t+1}\eta^{s-2}(\nabla_{xy}\varphi)^2\leq (1+p_0)\sum_{x,y\in V}\mu_{xy}u(x)^{-t+1}\varphi(x)^{s-2}(\nabla_{xy}\varphi)^2.\label{eq5}
\end{equation}
Letting $\epsilon=\frac{t}{p_0^2(1+p_0^2)}$, substituting \eqref{eq4} and \eqref{eq5} into \eqref{eq3-0}, we obtain
%
\begin{align}
&\frac{t}{4p_0^2}\sum_{x,y\in V}\mu_{xy}u(x)^{-t-1}\varphi(x)^s(\nabla_{xy}u)^2+\sum_{x\in V}\mu(x)\varphi(x)^{s}u(x)^{\sigma-t}\notag\\
&\leq\frac{C(p_0)s^2}{t}\sum_{x,y\in V}\mu_{xy}u(x)^{-t+1}\varphi(x)^{s-2}(\nabla_{xy}\varphi)^2,\label{eq7}
\end{align}
where $C(p_0)=p_0^2(1+p_0)(1+p_0^2)$.

By using Young's inequality with H\"{o}lder conjugates $p_1=\frac{\sigma-t}{1-t}$ and $p_2=\frac{\sigma-t}{\sigma-1}$, we estimate the RHS of \eqref{eq7} as follows.
\begin{align}
&\frac{C(p_0)s^2}{t}\sum_{x,y\in V}\mu_{xy}u(x)^{-t+1}\varphi(x)^{s-2}(\nabla_{xy}\varphi)^2\notag\\
=\ &\sum_{x,y\in V}\mu_{xy}\left(u(x)^{-t+1}\varphi(x)^{\frac{s}{p_1}}\right)\left(\frac{C(p_0)s^2}{t}\varphi(x)^{\frac{s}{p_2}-2}(\nabla_{xy}\varphi)^2\right)\notag\\
\leq\ &\varepsilon\sum_{x,y\in V}\mu_{xy}\left(u(x)^{-t+1}\varphi(x)^{\frac{s}{p_1}}\right)^{p_1}+C_{\varepsilon}\sum_{x,y\in V}\mu_{xy}\left(\frac{C(p_0)s^2}{t}\varphi(x)^{\frac{s}{p_2}-2}(\nabla_{xy}\varphi)^2\right)^{p_2}\notag\\
{\color{magenta}=}\ &\frac12\sum_{x,y\in V}\mu_{xy}u(x)^{\sigma-t}\varphi(x)^{s}+\frac C{t^{\frac{\sigma-t}{\sigma-1}}}\sum_{x,y\in V}\mu_{xy}\varphi(x)^{{s}-\frac{2(\sigma-t)}{\sigma-1}}(\nabla_{xy}\varphi)^{\frac{2(\sigma-t)}{\sigma-1}},\label{eq8}
\end{align}
where we have taken $\varepsilon=\frac12$ and used that $0\leq\varphi\leq1$.
The constants on $p_0$, $\sigma$ and $s$ are absorbed into $C$.

Substituting \eqref{eq8} into \eqref{eq7}, we obtain
\begin{align}
\frac{t}{4p_0^2}\sum_{x,y\in V}\mu_{xy}u(x)^{-t-1}\varphi(x)^s(\nabla_{xy}u)^2&+\frac12\sum_{x\in V}\mu(x)\varphi(x)^{s}u(x)^{\sigma-t}\notag\\
&\leq C{t^{\frac{\sigma-t}{1-\sigma}}}\sum_{x,y\in V}\mu_{xy}\varphi(x)^{{s}-\frac{2(\sigma-t)}{\sigma-1}}(\nabla_{xy}\varphi)^{\frac{2(\sigma-t)}{\sigma-1}}.\label{eq9}
\end{align}

We then turn back to \eqref{eq1}, take another $\psi=\varphi^s$, and get that
\begin{equation}\label{eq10}
-\frac12\sum_{x,y\in V}\mu_{xy}(\nabla_{xy}u)\nabla_{xy} (\varphi^s)+\sum_{x\in V}\mu(x)u(x)^{\sigma}\varphi(x)^s\leq0.
\end{equation}
By using \eqref{eq3'}, and from \eqref{eq10}, we obtain
\begin{align}
&\sum_{x\in V}\mu(x)u(x)^{\sigma}\varphi(x)^s\notag\\
\leq &\frac s2\sum_{x,y\in V}\mu_{xy}\eta^{s-1} (\nabla_{xy}u)(\nabla_{xy}\varphi)\notag\\
\leq&\frac s2\left(\sum_{x,y\in V}\mu_{xy}\eta^su(x)^{-t-1}(\nabla_{xy} u)^2\right)^{\frac12}\left(\sum_{x,y\in V}\mu_{xy}\eta^{s-2}u(x)^{t+1}(\nabla_{xy}\varphi)^2\right)^{\frac12},\label{eq11}
\end{align}
where the last inequality is by Cauchy-Schwarz.

Now similar to \eqref{eq4}, we have
\begin{equation}\label{eq12}
\sum_{x,y\in V}\mu_{xy}\eta^{s-2}u(x)^{t+1}(\nabla_{xy}\varphi)^2\leq (1+p_0^{2})\sum_{x,y\in V}\mu_{xy}u(x)^{t+1}\varphi(x)^{s-2}(\nabla_{xy}\varphi)^2.
\end{equation}
Substituting \eqref{eq4} and \eqref{eq12} into \eqref{eq11}, we see that
\begin{align}\label{eq13}
&\sum_{x\in V}\mu(x)u(x)^{\sigma}\varphi(x)^s\notag\\
\leq &\frac {s(1+p_0^2)}2\left(\sum_{x,y\in V}\mu_{xy}u(x)^{-t-1}\varphi(x)^s(\nabla_{xy}u)^2\right)^{\frac12}\left(\sum_{x,y\in V}\mu_{xy}u(x)^{t+1}\varphi(x)^{s-2}(\nabla_{xy}\varphi)^2\right)^{\frac12}.
\end{align}
Now we use H\"{o}lder's inequality with conjugates $p_3=\frac{\sigma}{t+1}$ and $p_4=\frac{\sigma}{\sigma-t-1}$ to obtain
\begin{align}\label{eq14}
&\sum_{x,y\in V}\mu_{xy}u(x)^{t+1}\varphi(x)^{s-2}(\nabla_{xy}\varphi)^2\notag\\
=&\sum_{x,y\in V,\nabla_{xy}\varphi\neq0}\mu_{xy}u(x)^{t+1}\varphi(x)^{s-2}(\nabla_{xy}\varphi)^2\notag\\
=&\sum_{x,y\in V,\nabla_{xy}\varphi\neq0}\mu_{xy}\left(\varphi(x)^{\frac{s}{p_3}}u(x)^{t+1}\right)\left(\varphi(x)^{\frac{s}{p_4}-2}(\nabla_{xy}\varphi)^2\right)\notag\\
\leq&\left(\sum_{x,y\in V,\nabla_{xy}\varphi\neq0}\mu_{xy}\varphi(x)^{{s}}u(x)^{\sigma}\right)^{\frac{t+1}{\sigma}}\left(\sum_{x,y\in V,\nabla_{xy}\varphi\neq0}\mu_{xy}\varphi(x)^{{s}-\frac{2\sigma}{\sigma-t-1}}
(\nabla_{xy}\varphi)^{\frac{2\sigma}{\sigma-t-1}}\right)^{\frac{\sigma-t-1}{\sigma}}\notag\\
\leq&\left(\sum_{x,y\in V,\nabla_{xy}\varphi\neq0}\mu_{xy}\varphi(x)^{{s}}u(x)^{\sigma}\right)^{\frac{t+1}{\sigma}}\left(\sum_{x,y\in V,\nabla_{xy}\varphi\neq0}\mu_{xy}(\nabla_{xy}\varphi)^{\frac{2\sigma}{\sigma-t-1}}\right)^{\frac{\sigma-t-1}{\sigma}},
\end{align}
where we have used $0\leq\varphi\leq1$ and ${s}>\frac{2\sigma}{\sigma-t-1}$.

By \eqref{eq9}, we have that
\begin{align}\label{eq15}
&\sum_{x,y\in V}\mu_{xy}u(x)^{-t-1}\varphi(x)^s(\nabla_{xy}u)^2\notag\\
\leq& C{t^{\frac{\sigma}{1-\sigma}-1}}\sum_{x,y\in V}\mu_{xy}\varphi(x)^{{s}-\frac{2(\sigma-t)}{\sigma-1}}(\nabla_{xy}\varphi)^{\frac{2(\sigma-t)}{\sigma-1}}\notag\\
\leq&C{t^{\frac{\sigma}{1-\sigma}-1}}\sum_{x,y\in V}\mu_{xy}(\nabla_{xy}\varphi)^{\frac{2(\sigma-t)}{\sigma-1}}.
\end{align}
Substituting \eqref{eq14} and \eqref{eq15} into \eqref{eq13}, we obtain
\begin{align}\label{eq16}
\sum_{x\in V}\mu(x)u(x)^{\sigma}\varphi(x)^s\leq &C{t^{\frac{2\sigma-1}{2(1-\sigma)}}}\left(\sum_{x,y\in V,\nabla_{xy}\varphi\neq0}\mu_{xy}\varphi(x)^{{s}}u(x)^{\sigma}\right)^{\frac{t+1}{2\sigma}}\notag\\
&\times\left(\sum_{x,y\in V}\mu_{xy}(\nabla_{xy}\varphi)^{\frac{2(\sigma-t)}{\sigma-1}}\right)^{\frac12}\left(\sum_{x,y\in V,\nabla_{xy}\varphi\neq0}\mu_{xy}(\nabla_{xy}\varphi)^{\frac{2\sigma}{\sigma-t-1}}\right)^{\frac{\sigma-t-1}{2\sigma}},
\end{align}
which implies that
\begin{align}\label{eq17}
&\left(\sum_{x\in V}\mu(x)u(x)^{\sigma}\varphi(x)^s\right)^{1-\frac{t+1}{2\sigma}}\notag\\
\leq& C{t^{\frac{2\sigma-1}{2(1-\sigma)}}}\left(\sum_{x,y\in V}\mu_{xy}(\nabla_{xy}\varphi)^{\frac{2(\sigma-t)}{\sigma-1}}\right)^{\frac12}\left(\sum_{x,y\in V}\mu_{xy}(\nabla_{xy}\varphi)^{\frac{2\sigma}{\sigma-t-1}}\right)^{\frac{\sigma-t-1}{2\sigma}}.
\end{align}

\bigskip

Step 2. We now estimate the $L^\sigma$-norm of $u$ via \eqref{eq17} by choosing a suitable sequence of $\varphi_i$'s.

Let $o\in V$ be as in \eqref{volume} and denote $d(x)=d(o,x)$ for brevity. For any integer $k\geq0$, we denote $B_k=\{x\in V:\ d(x)\leq 2^k\}$ and define the function $h_k$ on $V$ by
\begin{equation*}\label{hk}
h_k=\begin{cases}
1 & d(x)\leq 2^k,\\
2-\frac{d(x)}{2^k} & 2^k< d(x)<2^{k+1},\\
0 & d(x)\geq 2^{k+1}.
\end{cases}
\end{equation*}
For $i\geq1$, define
\begin{equation*}\label{phik}
\varphi_i(x)=\frac{1}{i}\sum_{k=i-1}^{2i-2}h_k(x).
\end{equation*}
Then it is clear that $0\leq\varphi_i\leq1$ and $\varphi_i=0$ in $\left(B_{2i-1}\right)^c$, $\varphi_i=1$ in $B_{i-1}$. Hence $\varphi_i\rightarrow1$, the identity function on $V$ as $i\rightarrow\infty$. Moreover, for any $x\in B_k-B_{k-1}$, if $k\leq i-2$ or $k\geq 2i+1$, then  $\nabla_{xy}\varphi_i=0$ for any $y\sim x$; if $i-1\leq k\leq 2i$, then for any $y\sim x$, we have
\begin{equation*}\label{gradphik}
|\nabla_{xy}\varphi_i|\lesssim \frac{1}{i\cdot2^k}.
\end{equation*}
We replace $\varphi$ in \eqref{eq17} by $\varphi_i$, and obtain
\begin{align}
&\sum_{x,y\in V}\mu_{xy}(\nabla_{xy}\varphi_i)^{\frac{2(\sigma-t)}{\sigma-1}}\notag\\
\leq&\sum_{k={i-1}}^{2i}\sum_{x\in B_k-B_{k-1}}\sum_{y\sim x} \mu_{xy}(\nabla_{xy}\varphi_i)^{\frac{2(\sigma-t)}{\sigma-1}}\notag\\
\leq& C\sum_{k={i-1}}^{2i}\sum_{x\in B_k-B_{k-1}}\mu(x)\left(\frac{1}{i\cdot2^k}\right)^{\frac{2(\sigma-t)}{\sigma-1}}\notag\\
\leq& C\sum_{k={i-1}}^{2i}\mu(B_k)\left(\frac{1}{i\cdot2^k}\right)^{\frac{2(\sigma-t)}{\sigma-1}}\notag\\
\leq& C\sum_{k={i-1}}^{2i}\left(2^k\right)^{\frac{2\sigma}{\sigma-1}}\left(\ln 2^k\right)^{\frac{1}{\sigma-1}}\left(\frac{1}{i\cdot2^k}\right)^{\frac{2(\sigma-t)}{\sigma-1}}\notag\\
\leq& C\sum_{k={i-1}}^{2i}2^{\frac{2kt}{\sigma-1}}i^{-\frac{2(\sigma-t)}{\sigma-1}}k^{\frac{1}{\sigma-1}}\notag\\
\leq & C'i^{\frac{2t-\sigma}{\sigma-1}}2^{\frac{4it}{\sigma-1}}.\label{eq18}
\end{align}
Similarly, we have
\begin{align}
&\sum_{x,y\in V}\mu_{xy}(\nabla_{xy}\varphi_i)^{\frac{2\sigma}{\sigma-t-1}}\notag\\
\leq&C\sum_{k={i-1}}^{2i}\left(2^k\right)^{\frac{2\sigma}{\sigma-1}}\left(\ln 2^k\right)^{\frac{1}{\sigma-1}}\left(\frac{1}{i\cdot2^k}\right)^{\frac{2\sigma}{\sigma-t-1}}\notag\\
\leq& C\sum_{k={i-1}}^{2i}2^{k(\frac{2\sigma}{\sigma-1}-\frac{2\sigma}{\sigma-t-1})}i^{-\frac{2\sigma}{\sigma-t-1}}k^{\frac{1}{\sigma-1}}\notag\\
\leq & C'i^{\frac{\sigma}{\sigma-1}-\frac{2\sigma}{\sigma-t-1}}2^{i(\frac{2\sigma}{\sigma-1}-\frac{2\sigma}{\sigma-t-1})}.\label{eq19}
\end{align}
Substituting \eqref{eq18} and \eqref{eq19} into \eqref{eq17}, we obtain
\begin{align}
&\left(\sum_{x\in V}\mu(x)u(x)^{\sigma}\varphi_i(x)^s\right)^{1-\frac{t+1}{2\sigma}}\notag\\
\leq& C t^{\frac{2\sigma-1}{2(1-\sigma)}}\left(i^{\frac{2t-\sigma}{\sigma-1}}2^{\frac{4it}{\sigma-1}}\right)^{\frac12}
\left(i^{\frac{\sigma}{\sigma-1}-\frac{2\sigma}{\sigma-t-1}}2^{i(\frac{2\sigma}{\sigma-1}-\frac{2\sigma}{\sigma-t-1})}\right)^{\frac{\sigma-t-1}{2\sigma}}\notag\\
\leq& C\left(i^{\frac{1}{i}}\right)^{\frac{1}{2(\sigma-1)}},\label{eq20}
\end{align}
in which, we set $t=\frac{1}{i}$.
Note that the RHS of \eqref{eq20} is uniformly bounded in $i$, we obtain by letting $i\rightarrow\infty$ that
\begin{equation*}\label{eq21}
\sum_{x\in V}\mu(x)u(x)^{\sigma}<\infty.
\end{equation*}
This implies that as $i\rightarrow\infty$,
\begin{equation*}\label{eq22}
\sum_{x,y\in V,\nabla_{xy}\varphi_i\neq0}\mu_{xy}\varphi_i(x)^su(x)^{\sigma}\rightarrow0.
\end{equation*}
Substituting the above $\varphi_i$ and $t=1/i$ back into \eqref{eq16} and repeating the same procedure as above, we see that as $i\rightarrow\infty$,
\begin{equation*}\label{eq23}
\sum_{x\in V}\mu(x)u(x)^{\sigma}\varphi_i(x)^s\rightarrow0,
\end{equation*}
which implies
\begin{equation*}\label{eq23}
\sum_{x\in V}\mu(x)u(x)^{\sigma}=0,
\end{equation*}
and hence $u\equiv0$, which is a contradiction to the assumption that $u>0$. Therefore, there exists no positive solution to equation \eqref{Elliptic}.
\end{proof}

\bigskip

\section{An example}
In this section, we provide an example of weighted graph to see that the indices $\frac{2\sigma}{\sigma-1}$ and $\frac{1}{\sigma-1}$ in \eqref{volume} are sharp.

Let $N\geq2$, we introduce the notion of homogeneous tree $\mathbb T_N$. We say that a connected graph $(V,E)$ is a tree if for any two distinct points $x,y\in V$, there is only one path between $x$ and $y$. A homogeneous tree $\mathbb T_N$ is defined to be a tree where all vertices have degree $N$. The following example is on $\mathbb T_N$. Fix an arbitrary vertex $o\in \mathbb T_N$ as the root. For $n\geq0$, we denote by $D_n:\ =\{x\in \mathbb T_N:\ d(o,x)=n\}$ the collection of all the vertices with distance $n$ from $o$, and denote by $E_n$ the collection of all the edges from vertices in $D_n$ to vertices in $D_{n+1}$.
\begin{theorem}
Let $(V,E)=\mathbb T_N$. For any arbitrary small $\varepsilon>0$, there exists a weight $\mu$ on $\mathbb T_N$ such that $\mu(B(o,n))\asymp n^{\frac{2\sigma}{\sigma-1}}(\ln n)^{\frac{1}{\sigma-1}+\varepsilon}$ for $n\geq2$ and for such $(\mathbb T_N,\mu)$, there exists a solution $u$ to \eqref{Elliptic}. Moreover, we can take $\mu$ and $u$ as following:

\begin{eqnarray*}
\mu_{xy}&=\mu_{n}&=\frac{(n+n_0)^{\frac{\sigma+1}{\sigma-1}}(\ln (n+n_0))^{\frac{1}{\sigma-1}+\varepsilon}}{(N-1)^n}\qquad  \text{for any $(x,y)\in E_n,\ n\geq0$},\\
u(x)&=u_n&= \frac{\delta}{(n+n_0)^{\frac{2}{\sigma-1}}(\ln (n+n_0))^{\frac{1}{\sigma-1}}}\qquad\qquad\text{for any $x\in D_n,\ n\geq0$},
\end{eqnarray*}
where $n_0\geq2$ is sufficiently large and $\delta>0$ is sufficiently small.
\end{theorem}
\begin{proof}
First we show that under the above weight, the volume satisfies
\begin{eqnarray*}
\mu(B(o,n))\asymp n^{\frac{2\sigma}{\sigma-1}}(\ln n)^{\frac{1}{\sigma-1}+\varepsilon}.
\end{eqnarray*}
Indeed, for $n\geq2$, we compute
\begin{equation*}
\mu(B(o,n))=\sum_{k=0}^n\mu(D_k)\asymp\sum_{k=0}^n(N-1)^k\mu_k\asymp n^{\frac{2\sigma}{\sigma-1}}(\ln n)^{\frac{1}{\sigma-1}+\varepsilon}.
\end{equation*}
Then we check that \eqref{Elliptic} holds for the $\mu$ and $u$ given as above, that is
\begin{eqnarray}
u_1-u_0+u_0^\sigma&\leq&0,\label{eqZ0}\\
\frac{(N-1)\mu_{n}u_{n+1}+\mu_{n-1}u_{n-1}}{(N-1)\mu_{n}+\mu_{n-1}}-u_n+u_n^\sigma&\leq&0,\qquad\text{for all $n\geq 1$.}\label{eqZ}
\end{eqnarray}

The two constants $N$ and $\delta$ are closely related to $\varepsilon$ and we will first determine $n_0$ and then $\delta$. For brevity, we denote $p=\frac{1}{\sigma-1}$.  The proof of \eqref{eqZ} will be separated into the following two cases $n=0$ and $n\neq0$.

{\bf Case 1.} $n=0$.
By substituting the values of $u$, \eqref{eqZ0} is equivalent to
\begin{equation*}
\frac{\delta}{(n_0+1)^{2p}(\ln(n_0+1))^p}-\frac{\delta}{n_0^{2p}(\ln n_0)^p}+\left(\frac{\delta}{n_0^{2p}(\ln n_0)^p}\right)^\sigma\leq0,
\end{equation*}
which is satisfied by choosing $\delta\leq\delta_0$ with
\begin{equation*}
\delta_0=n_0^{2p}(\ln n_0)^p\left(\frac{(n_0+1)^{2p}(\ln(n_0+1))^p-n_0^{2p}(\ln n_0)^p}{(n_0+1)^{2p}(\ln(n_0+1))^p}\right)^{\frac1{\sigma-1}}.
\end{equation*}

\medskip

{\bf Case 2.} $n\geq1$. By substituting explicit values of $\mu$ and $u$, we see that \eqref{eqZ} is equivalent to
\begin{align}\label{eqZp}
&\frac{\delta\cdot\frac{(n+n_0)^{2p+1}(\ln (n+n_0))^{p+\varepsilon}}{(n+n_0+1)^{2p}(\ln (n+n_0+1))^{p}}+\delta\cdot\frac{(n+n_0-1)^{2p+1}(\ln (n+n_0-1))^{p+\varepsilon}}{(n+n_0-1)^{2p}(\ln (n+n_0-1))^{p}}}{(n+n_0)^{2p+1}(\ln (n+n_0))^{p+\varepsilon}+(n+n_0-1)^{2p+1}(\ln (n+n_0-1))^{p+\varepsilon}}
-\frac{\delta}{(n+n_0)^{2p}(\ln (n+n_0))^{p}}\notag\\&+\left(\frac{\delta}{(n+n_0)^{2p}(\ln (n+n_0))^{p}}\right)^\sigma\leq0.
\end{align}
Observing that $\sigma=\frac{p+1}{p}$, thus \eqref{eqZp} is equivalent to that $\delta$ satisfies
\begin{align}
&\delta^{\sigma-1}\leq (n+n_0)^{2}\ln (n+n_0)-(n+n_0)^{2p+2}(\ln (n+n_0))^{p+1}\cdot\notag\\
&\frac{\frac{(n+n_0)^{2p+1}(\ln (n+n_0))^{p+\varepsilon}}{(n+n_0+1)^{2p}(\ln (n+n_0+1))^{p}}+\frac{(n+1)^{2p+1}(\ln (n+n_0-1))^{p+\varepsilon}}{(n+n_0-1)^{2p}(\ln (n+1))^{p}}}{(n+n_0)^{2p+1}(\ln (n+n_0))^{p+\varepsilon}+(n+n_0-1)^{2p+1}(\ln (n+n_0-1))^{p+\varepsilon}}.\label{eqZdelta}
\end{align}
By some amount of computation, we have that
\begin{equation}\label{eqZlimit}
\lim_{n\rightarrow\infty}n^{2}\ln n-n^{2p+2}(\ln n)^{p+1}\cdot\frac{\frac{n^{2p+1}(\ln n)^{p+\varepsilon}}{(n+1)^{2p}(\ln (n+1))^{p}}+\frac{(n-1)^{2p+1}(\ln (n-1))^{p+\varepsilon}}{(n-1)^{2p}(\ln (n-1))^{p}}}{n^{2p+1}(\ln n)^{p+\varepsilon}+(n-1)^{2p+1}(\ln (n-1))^{p+\varepsilon}}
=p\varepsilon.
\end{equation}
(The details are as follows.
we will use $\frac{\ln (n-1)}{\ln n}=1-\frac{1}{n\ln n}-\frac{1}{2n^2\ln n}+o(\frac{1}{n^2\ln n})$ frequently.
The function of $n$ under the limit on the LHS of \eqref{eqZlimit} can be written as
\begin{equation*}
n^{2}\ln n\left(1-\frac{\left(1-\frac{1}{n+1}\right)^{2p}\left(\frac{\ln n}{\ln(n+1)}\right)^p+\left(1-\frac{1}{n}\right)\left(\frac{\ln (n-1)}{\ln n}\right)^{\varepsilon}}{1+\left(1-\frac{1}{n}\right)^{2p+1}\left(\frac{\ln (n-1)}{\ln n}\right)^{p+\varepsilon}}\right),
\end{equation*}
where we use $\left(1-\frac{1}{n}\right)^{\alpha}=1-\frac{\alpha}{n}-\frac{\alpha(\alpha-1)}{2n^2}$ with $\alpha>0$, then the above Taylor expansion for $\frac{\ln (n-1)}{\ln n}$ gives
\begin{align*}
&n^{2}\ln n\left(1-\frac{2-\frac{2p}{n+1}-\frac{1}{n}-\frac{p}{(n+1)\ln (n+1)}-\frac{\varepsilon}{n\ln n}+\frac{2p(2p-1)}{2(n+1)^2}+\frac{4p^2-p}{2(n+1)^2\ln (n+1)}+\frac{\varepsilon}{2n^2\ln n}}{2-\frac{2p+1}{n}-\frac{p+\varepsilon}{n\ln n}+\frac{2p(2p+1)}{2n^2}-\frac{p+\varepsilon}{2n^2\ln n}+\frac{(2p+1)(p+\varepsilon)}{n^2\ln n}}\right)+o(1)\\
&=n^{2}\ln n\left(\frac{\frac{2p\varepsilon}{n^2\ln n}}{2}\right)+o(1)\\
&=p\varepsilon+o(1),
\end{align*}
where $o(1)\rightarrow0$ as $n\rightarrow\infty$, and hence proves \eqref{eqZlimit}.)

This implies that there exists some large $n_0$ such that for all $n\geq 0$, the RHS of \eqref{eqZdelta} is
bounded from above by $\delta_1:=\ \frac{p\varepsilon}{2}$.

Finally, take $n_0$ as above and $\delta=\min\{\delta_0,\delta_1\}$, we see that $u$ is a solution to \eqref{Elliptic}.
\end{proof}


\begin{thebibliography}{99}
%
%
%

%
%
\bibitem{CM}\textsc{F. Camilli, C. Marchi}, \emph{A note on Kazdan-Warner equation on networks}, Advances in Calculus of Variations (2020),  DOI: 10.1515/ACV-2020-0046.


\bibitem{CY} \textsc{S.Y. Cheng and S.-T. Yau} \emph{
Differential equations on Riemannian manifolds and their geometric applications}, Comm.
Pure Appl. Math. \textbf{28} (1975) 333-354.

%
\bibitem{Ge}\textsc{H. Ge}, \emph{A-th Yamabe equation on graph}, Proc. Amer. Math. Soc. \textbf{146} (2018), no. 5, 2219-2224.


\bibitem{GHJ}\textsc{H. Ge, B. Hua, W. Jiang}, \emph{A note on Liouville type equations on graphs},  Proc. Amer. Math. Soc. \textbf{146} (2018), no. 11, 4837-4842.

\bibitem{G85} \textsc{A. Grigor'yan}, \emph{On the existence of positive
fundamental solution of the Laplace equation on Riemannian manifolds},
Matem. Sb. \textbf{128} (1985) 354-363. English transl. Math. USSR Sb.
\textbf{56} (1987) 349-358.


\bibitem{G} \textsc{A. Grigor'yan}, \emph{\ Heat Kernel and Analysis on
Manifolds}, AMS/IP Studies in Adv. Math. \textbf{47}, 2009.


%
%
\bibitem{GLY1} \textsc{A. Grigor'yan, Y. Lin, Y. Yang}, \emph{Yamabe type equations on graphs},
J. Differential Equations  \textbf{261} (2016), no. 9, 4924-4943.

\bibitem{GLY2} \textsc{A. Grigor'yan, Y. Lin, Y. Yang}, \emph{ Kazdan-Warner equations on graphs},
 Calc. Var. Partial Differential Equations 55 (2016), no. 4, Art. 92, 13 pp.

\bibitem{GLY3} \textsc{A. Grigor'yan, Y. Lin, Y. Yang},  \emph{Existence of positive solutions to some nonlinear equations on locally finite graphs}, Sci. China Math. \textbf{60}(2017), no. 7, 1311-1324.

\bibitem{GS1} \textsc{A. Grigor'yan and Y. Sun,} \emph{On nonnegative
solutions of the inequality $\Delta u+u^{\sigma} \le 0$ on Riemannian
manifolds}, Comm. Pure Appl. Math. \textbf{67} (2014), 1336-1352.


\bibitem{GSV}\textsc{A. Grigor'yan, Y. Sun, and I. E. Verbitsky,} \emph{Superlinear elliptic inequalities on manifolds}, J. Funct. Anal. \textbf{278}(2020),108444.


\bibitem{HSZ}\textsc{X. Han, M. Shao, L. Zhao}, \emph{Existence and convergence of solutions for nonlinear biharmonic equations on graphs},
  J. Differential Equations 268 (2020), no. 7, 3936-3961.


\bibitem{HS} \textsc{X. Han, M. Shao}, \emph{p-Laplacian Equations on Locally Finite Graphs}, Acta Mathematica Sinica \textbf{37}(2021), Issue: 11, 1-34.

\bibitem{HWY} \textsc{H.-Y. Huang, J. Wang, W. Yang}, \emph{Mean field equation and relativistic Abelian Chern-Simons model on finite graphs}, J. Funct. Anal. \textbf{281}(2021), no. 10, Paper No. 109218.

\bibitem{K}\textsc{L. Karp}, \emph{Subharmonic functions, harmonic mappings and isometric immersions}, in: ``Seminar
on Differential Geometry'', ed. S.T.Yau, Ann. Math. Stud. \textbf{102}, Princeton, 1982.


 \bibitem{LW1}\textsc{Y. Lin, Y. Wu}, \emph{The existence and nonexistence of global solutions
for a semilinear heat equation on graphs},  Calc. Var. Partial Differential Equations \textbf{56}(2017), no. 4, Paper No. 102, 22 pp.


\bibitem{LW2} \textsc{Y. Lin, Y. Wu}, \emph{Blow-up problems for nonlinear parabolic equations on locally finite graphs},
Acta Math. Sci. Ser. B (Engl. Ed.) \textbf{38} (2018), no. 3, 843-856.


\bibitem{LY} \textsc{Y. Lin, Y. Yang}, \emph{A heat flow for the mean field equation on a finite graph}, Calc. Var. Partial Differential Equations
\textbf{60} (2021), no. 6, Paper No. 206, 15 pp.

\bibitem{LiuY}\textsc{S. Liu, Y. Yang}, \emph{Multiple solutions of Kazdan-Warner equation on graphs in the negative case},
  Calc. Var. Partial Differential Equations \textbf{59}(2020), no. 5, Paper No. 164, 15 pp.


%

%
%


%
%

%


\bibitem{MP} \textsc{E. Mitidieri and S. I. Pohozaev}, \emph{Nonexistence of
positive solutions for quasilinear elliptic problems on $\mathbb{R}^{N}$},
Proc. Steklov Inst. Math. \textbf{227} (1999), 186-216.


\bibitem{SZ} \textsc{J. Serrin and H. Zou}, \emph{Cauchy-Liouville and
universal boundedness theorems for quasilinear elliptic equations and
inequalities}, Acta. Math. \textbf{189}(2002), 79-142.



\bibitem{V} \textsc{N. Varopoulos,} \emph{Potential theory and diffusion
on Riemannian manifolds,} in Conf. on Harmonic Analysis in Honor of A.
Zygmund, Wadsworth Math. Series, Wadsworth, Belmont, CA, pp. 821-837, 1983.

\bibitem{W} \textsc{W. Woess,} \emph{Random walks on infinite graphs and groups,} Cambridge U. Process, Cambridge, 2000.
%
\end{thebibliography}
\end{document}